\documentclass[12pt]{article}

\usepackage[margin=1in]{geometry}
\usepackage{amsmath,amsthm,amssymb,amsfonts,amscd,graphicx,color,xcolor,mathrsfs,latexsym}
\usepackage{subfig}
\usepackage{array}
\usepackage{tikz}
\usetikzlibrary{arrows}

\usepackage[bookmarksnumbered,colorlinks,plainpages]{hyperref}
\hypersetup{
    colorlinks=true,
    linkcolor=red,
    anchorcolor=green,
    citecolor=cyan,
    urlcolor=blue,
    filecolor=magenta,
    pdftoolbar=true
}

\newcommand{\Z}{\mathcal{Z}}
\newcommand{\A}{\mathcal{A}}
\newcommand{\C}{\mathcal{C}}

\newcommand{\M}{\mathcal{M}}

\renewcommand{\leq}{\leqslant}
\renewcommand{\geq}{\geqslant}

\newcommand{\MP}{\mathcal{MP}}

\numberwithin{equation}{section}

\newtheorem{theorem}{Theorem}[section]
\newtheorem{lemma}[theorem]{Lemma}
\newtheorem{proposition}[theorem]{Proposition}
\newtheorem{corollary}[theorem]{Corollary}
\theoremstyle{definition}

\newtheorem{remark}[theorem]{Remark}

\title{Riordan Arrays and Shifted Hankel Determinants of OEIS \href{https://oeis.org/A005773}{A005773}}

\author{
    {\bf Daniel Yaqubi}\thanks{Department of Computer Science, University of Torbat-e Jam, Torbat-e Jam, Iran. \mbox{E-mail}: \texttt{yaqubi@tjamcaas.ac.ir}, \texttt{daniel\_yaqubi@yahoo.es}}
    \and
    {\bf Madjid Mirzavaziri}\thanks{Department of Pure Mathematics, Ferdowsi University of Mashhad, Mashhad, Iran. \mbox{E-mail}: \texttt{mirzavaziri@um.ac.ir}, \texttt{mirzavaziri@gmail.com}}
}

\date{}

\begin{document}

\maketitle

\let\thefootnote\relax
\addtocounter{footnote}{-1}

\begin{abstract}
We study two Riordan arrays associated with the Motzkin-prefix sequence (\href{https://oeis.org/A005773}{OEIS A005773}). By connecting their production matrices through a simple boundary factorization, we obtain a period-three transformation on Motzkin endpoint distributions and exact formulas for paths by final height modulo~$3$. In addition, a unified factorization for the shifted Hankel matrices leads directly to explicit closed forms for the third and fourth Hankel determinants, $\Delta_n^{(3)}$ and~$\Delta_n^{(4)}$.
\end{abstract}

\section{Introduction}

Motzkin paths are fundamental objects in enumerative combinatorics.
In this paper we study their prefix version.  A \emph{Motzkin prefix}
of length $n$ is a lattice path starting at $(0,0)$, with steps $\mathbf{S}=\{u=(1,1), h=(1,0), d=(1,-1)\},$
which never goes below the $x$-axis.  We denote by $\MP_n$ the number
of Motzkin prefixes of length $n$. Thus $\MP_n=1,2,5,13,35,96,267,750,2123,\ldots$ for $n\geq 0$.
We shall also consider the closely related sequence
\[\C=\{a_n\}_{n\geq0}=\{1,1,2,5,13,35,96,267,750,2123,\ldots\},\]
which is OEIS \href{https://oeis.org/A005773}{OEIS A005773}.  Sloane's database records well combinatorial models for this sequence \cite{A005773}; the two most useful for us are first
 \emph{Directed animals} of size $n$ (Dhar et al.\ \cite{Dhar1982}; Bousquet-M\'elou \cite{BousquetMelou1998}), and the second is
left factors of length $n$ of Motzkin paths, which is lattice paths from $(0,0)$ using steps $\mathbf{S}$ that never go below the $x$-axis, with \emph{no} requirement of returning to height $0$ (Deutsch \cite{A005773}). This interpretation also used by Yaqubi and et.al,\cite{Yaqubi2023}, who call the numbers \emph{perfect lattice numbers} and write $\C_n$ for $a(n)$. The OEIS entry also records its relation with binomial transforms, Catalan numbers, and the Riordan array \href{https://oeis.org/A038622}{OEISA038622}. We write $a_0=1,a_1=1,a_2=2,\dots$ throughout, matching the OEIS sequence .These two sequences are related by $a_0=1, a_{n+1}=\MP_n\quad(n\geq0).$ This shift is important, since the two sequences exhibit different
Hankel-determinant behaviour. The ordinary generating function of \href{https://oeis.org/A005773}{OEIS A005773} is
\[A(x):=\sum_{n\geq0}c_nx^n=\frac12+\frac12\sqrt{\frac{1+x}{1-3x}},\]
whereas the ordinary generating function of the Motzkin-prefix
sequence is
\[g(x):=\sum_{n\geq0}\MP_nx^n=\frac{C(x)-1}{x}=\frac{1}{2x}\left(\sqrt{\frac{1+x}{1-3x}}-1\right).\]
Thus $\C(x)=1+xg(x).$
The sequences above and their determinant properties have appeared
from several complementary viewpoints, including lattice paths inside
tables, operator calculus, Hankel determinants of path-generating
functions, and transfer-matrix methods.  In particular, the Hankel
determinant of the shifted sequence $\MP_n$
is identically equal to $1$, whereas the principal Hankel transform of
\href{https://oeis.org/A005773}{OEIS A005773} is periodic with period $6$ \cite{KrattenthalerYaqubi2018, Yaqubi2023}.  These two phenomena motivate
the comparison of the corresponding Riordan arrays.
Recall that an infinite lower-triangular array $D=(d_{n,k})_{n,k\ge 0}$ is a \emph{Riordan array} $D=(g,f)$ if $g(0)=1$, $f(0)=0$, $f'(0)\neq 0$, and the $k$-th column has generating function $g(x)f(x)^k$, i.e.\ $d_{n,k}=[x^n]\,g(x)f(x)^k$. The important feature is that the two arrays have the same $f$ in their  Riordan arrays but with different
first components.  We shall show that their difference is completely
described by a simple Riordan factor $R=\left(\frac{1}{1+x+x^2},x\right).$

We use the standard theory of the \emph{$\A$-sequence} and \emph{$\Z$-sequence} of a Riordan array \cite{Barry2017,Shapiro1991}. Our aim is to place the Motzkin-prefix sequence and \href{https://oeis.org/A005773}{OEIS A005773} in this
framework.  We determine the corresponding Riordan arrays, their production matrices and inverse arrays, and the associated Jacobi
structures.  The principal structural result is the factorization $DD^T=\bigl(\MP_{i+j}\bigr)_{i,j\geq0}$ for the Motzkin-prefix Riordan array.  This gives a direct
Riordan-array proof that $\det_{0\leq i,j<n}\bigl(\MP_{i+j}\bigr)=1, (n\geq1).$

\section{The Generating Functions}
\label{sec:GF}
We first record the following differential equations and coefficient
recurrences, which also appear in the OEIS entry for \href{https://oeis.org/A005773}{OEIS A005773}\cite{A005773, Yaqubi2023}.

\begin{lemma}\label{Som}
The functions $g(x)$ and $A(x)$ satisfies at $x(1+x)(1-3x)g'(x)+(1-4x-3x^2)g(x)-1=0,$ and $(1-2x-3x^2)A'(x)-2A(x)+1=0$.
Consequently, $(n+1)\MP_n=2(n+1)\MP_{n-1}+3(n-1)\MP_{n-2},$ and
$na_n=2na_{n-1}+3(n-2)a_{n-2},\quad n\geq2.$
\end{lemma}

\begin{proof}
Put $R(x)=\sqrt{\frac{1+x}{1-3x}}.$ Then $R(x)=1+2xg(x).$ Moreover,  $\frac{R'(x)}{R(x)}=\frac{2}{(1+x)(1-3x)}.$ Hence
$R'(x)=\frac{2R(x)}{(1+x)(1-3x)}.$ Differentiating $R(x)=1+2xg(x)$ gives $2g(x)+2xg'(x)=\frac{2(1+2xg(x))}{(1+x)(1-3x)}.$
After simplification, $x(1+x)(1-3x)g'(x)+(1-4x-3x^2)g(x)-1=0.$ 
Now use $A(x)=1+xg(x).$ Thus $g(x)=\frac{A(x)-1}{x}$ and $g'(x)=\frac{xA'(x)-A(x)+1}{x^2}.$
Substituting these expressions into the differential equation for $g(x)$ and simplifying gives result. Now extract the coefficients in  $g(x)=\sum_{n\geq0}\MP_nx^n$ and $A(x)=\sum_{n\geq0}a_nx^n$ concluded proof.
\end{proof}
\begin{theorem}
\label{thm:GF}
Consider sequences $\C=\{a_n\}_{n\geq0}$ and $\MP_n.$
Thus
{\scriptsize
\[
A(x) := \sum_{n\ge0} a_n x^n = \frac{1}{2}\left(1+\sqrt{\frac{1+x}{1-3x}}\right), \quad
g(x) := \sum_{n\ge0} \MP_n x^n = \frac{1}{2x}\left(\sqrt{\frac{1+x}{1-3x}}-1\right)
\]
}
Equivalently, $A(x)=1+xg(x).$
\end{theorem}

\begin{proof}
Using Lemma \eqref{Som}, it is enough to solve the equation $(1-2x-3x^2)A'(x)-2A(x)+1=0$. Put $R(x):=\sqrt{\frac{1+x}{1-3x}}.$ Then
\[
\frac{R'(x)}{R(x)}
=
\frac{2}{(1+x)(1-3x)}
=
\frac{2}{1-2x-3x^2}.
\]
Thus $(1-2x-3x^2)R'(x)=2R(x).$
Define $\widetilde A(x):=\frac{1+R(x)}2.$
Then
\[
(1-2x-3x^2)\widetilde A'(x)=R(x)=2\widetilde A(x)-1.
\]
Therefore $(1-2x-3x^2)\widetilde A'(x)-2\widetilde A(x)+1=0.$
Also, $\widetilde A(0)=1.$ Hence $\widetilde A(x)$ satisfies the same differential equation and
initial condition as $A(x)$. Finally, since $\MP_n=a_{n+1},$
we have
\[
g(x)=\sum_{n\geq0}\MP_nx^n=\sum_{n\geq0}a_{n+1}x^n=\frac{A(x)-a_0}{x}=\frac{A(x)-1}{x}.
\]
Substituting the expression for $A(x)$ gives proof.
\end{proof}
Also, using  Howroyd's identity, recorded at \cite{A005773}, the generating function sequence \href{https://oeis.org/A005773}{OEIS A005773} is equals $1/(1-xM(x))$, where
 \[
M(x)=\sum_{n\ge0}M_nx^n=\dfrac{1-x-\sqrt{1-2x-3x^2}}{2x^2}\]
 be the ordinary Motzkin generating function, and set $f(x)=xM(x)$. A short algebraic derivation from $M(x)=1+xM(x)+x^2M(x)^2$ (the classical Motzkin functional equation) gets, $f(x)=x\bigl(1+f(x)+f(x)^2\bigr),$
and solving the resulting quadratic in $f$ with $f(0)=0$ gives 
\[f(x)=\dfrac{1-x-\sqrt{1-2x-3x^2}}{2x}.\]
Substituting into $A(x)=1/(1-f(x))$ and rationalizing the denominator produces $\tfrac12+\tfrac12\sqrt{(1+x)/(1-3x)}$.

\begin{corollary}
\label{thm:motzkin-relation}
We have $a(n+1)=3\,a(n)-M(n-1), \quad n \ge 1,$
where $\M_n$ are the Motzkin numbers (\href{https://oeis.org/A001006}{OEIS A001006}).
\end{corollary}
\begin{proof}
See \cite[Lemma~2.6]{Yaqubi2023, A005773}.
\end{proof}
\begin{theorem}
\label{thm:riordan}
The two pairs $D_{\MP}:=(g(x),f(x))$ and $\widehat D:=(A(x),f(x))$
are proper Riordan arrays of the sequences $\MP_n$ and $a_n$, receptively.  The first column of $D_{\MP}$ is  $\MP_0,\MP_1,\MP_2,\ldots$, whereas the first column of
 $\widehat D$ is $a_0,a_1,a_2,\ldots.$Entries begin with
{\tiny
\[
\widehat D = \begin{pmatrix}
1 & 0 & 0 & 0 & 0 & 0 & \cdots \\
1 & 1 & 0 & 0 & 0 & 0 & \cdots \\
2 & 2 & 1 & 0 & 0 & 0 & \cdots \\
5 & 5 & 3 & 1 & 0 & 0 & \cdots \\
13 & 13 & 9 & 4 & 1 & 0 & \cdots \\
35 & 35 & 26 & 14 & 5 & 1 & \cdots \\
96 & 96 & 75 & 45 & 20 & 6 & \cdots \\
267 & 267 & 216 & 140 & 71 & 27 & \cdots
\end{pmatrix},
\quad
D_{\MP} = \begin{pmatrix}
1 & 0 & 0 & 0 & 0 & 0 & 0 & \cdots \\
2 & 1 & 0 & 0 & 0 & 0 & 0 & \cdots \\
5 & 3 & 1 & 0 & 0 & 0 & 0 & \cdots \\
13 & 9 & 4 & 1 & 0 & 0 & 0 & \cdots \\
35 & 26 & 14 & 5 & 1 & 0 & 0 & \cdots \\
96 & 75 & 45 & 20 & 6 & 1 & 0 & \cdots \\
267 & 216 & 140 & 71 & 27 & 7 & 1 & \cdots \\
750 & 623 & 427 & 238 & 105 & 35 & 8 & \cdots
\end{pmatrix}.
\]
}
\end{theorem}

\begin{proof}
By Theorem~\ref{thm:GF}, $g(0)=1$ and $A(0)=1.$ Furthermore, $f(x)=xM(x)=x+x^2+2x^3+4x^4+\cdots,$
so $f(0)=0,$ and $f'(0)=1\neq0.$ Therefore both pairs $(g,f) \text{and} (A,f)$ are proper Riordan arrays.
By definition, the first column of $(g,f)$ is $[x^n]g(x)=\MP_n,$ while the first column of $(A,f)$ is $[x^n]A(x)=a_n.$
Thus the asserted first-column sequences follow.
\end{proof}

\begin{lemma}
\label{prop:two-riordan-arrays}
With the notation above, $A(x)=1+xg(x)$ and $A(x)f(x)=xg(x).$ Consequently, if put $D_{\MP}=(d_{n,k})_{n,k\geq0}$
and $\widehat D=(\widehat d_{n,k})_{n,k\geq0},$ then for $n,k\geq1$ yields $\widehat d_{n,k}=d_{n-1,k-1}.$
\end{lemma}

\begin{proof}
The first and second identity follow from Theorem~\ref{thm:GF}.  For $n,k\geq1$, we can write
\[
\widehat d_{n,k}
=[x^n]A(x)f(x)^k
=[x^{n-1}]g(x)f(x)^{k-1}
=d_{n-1,k-1}.
\]
\end{proof}

Paul Barry \cite{Barry2017}, using Riordan array $D$, define two power series $\mathcal A(y)=\sum_{j\ge0}A_jy^j,$ and $\mathcal Z(y)=\sum_{j\ge0}Z_jy^j,$ which 
called them $\A$- and $\Z$-sequence generating functions. They are characterized by $f(x)=x\mathcal A(f(x))$ and $g(x)=\frac{g(0)}{1-x\mathcal Z(f(x))}.$
Consequently, $d_{n+1,k}=\sum_{j\ge0}A_j\,d_{n,k-1+j},$ and $d_{n+1,0}=\sum_{j\ge0}Z_j\,d_{n,j}.$ For the classical Motzkin arrays, we have $f(x)=x(1+f(x)+f(x)^2).$
Hence the common $\A$-sequence is $\mathcal A(y)=1+y+y^2.$ Moreover, for the sequence $\C=\{a_n\}_{n\geq0}$, since $A(x)=\frac{1}{1-f(x)},$
and therefore the $\Z$-sequence of $\widehat D=(A,f)$ coincides with its $\A$-sequence. 

\begin{proposition}
For $\widehat D=(A,f)$, we have $\mathcal A(y)= \mathcal Z(y)=1+y+y^2.$
For $D_{\MP}$, $\mathcal A(y)=1+y+y^2,$ and $\mathcal Z(y)=2+y.$
\end{proposition}

\begin{proof}
Since $f=x(1+f+f^2),$
the $A$-sequence of both arrays is $\A(y)=1+y+y^2.$ For $\widehat D$, since $A(x)=\frac1{1-f(x)}.$ Therefore $A(x)=\frac1{1-x(1+f+f^2)}.$ Thus $x\mathcal Z(f)=x(1+f+f^2),$
so $\mathcal Z(y)=1+y+y^2.$ For $D_{\MP}$, the identity $g=\frac1{1-2x-xf}$ gives $g=\frac1{1-x(2+f)},$
and hence $\mathcal Z(y)=2+y.$
\end{proof}

\begin{remark}
This distinction is the central structural feature:
\[
\begin{array}{c|cc}
&\mathcal A(y)&\mathcal Z(y)\\ \hline
D_{\MP}&1+y+y^2&2+y\\
\widehat D&1+y+y^2&1+y+y^2.
\end{array}
\]
Thus the two arrays have identical interior transition data and differ only
in their boundary data.
\end{remark}
%
Let $\overline D$ denote the matrix obtained from a Riordan array $D$
by deleting its first row. Following Barry's convention \cite{Barry2017}, the production
matrix of $D$ is $P_D=D^{-1}\overline D.$
\begin{theorem}
\label{thm:production}
Let $D_{\MP}=(d_{n,k})_{n,k\ge0},$ and $\widehat D=(\widehat d_{n,k})_{n,k\ge0}.$
Then, with $P_{\MP} := D_{\MP}^{-1}\overline{D}_{\MP}$, and $P_{\mathcal{C}} := \widehat{D}^{-1}\overline{\widehat{D}}$,the production matrices are
{\tiny
\begin{align*}
P_{\widehat D}=
\begin{pmatrix}
1&1&0&0&0&\cdots\\
1&1&1&0&0&\cdots\\
1&1&1&1&0&\cdots\\
0&0&1&1&1&\ddots\\
0&0&0&1&1&\ddots\\
\vdots&\vdots&\vdots&\ddots&\ddots&\ddots
\end{pmatrix}
, 
P_{\MP}=
\begin{pmatrix}
2&1&0&0&0&\cdots\\
1&1&1&0&0&\cdots\\
0&1&1&1&0&\cdots\\
0&0&1&1&1&\ddots\\
0&0&0&1&1&\ddots\\
\vdots&\vdots&\vdots&\ddots&\ddots&\ddots
\end{pmatrix},
\end{align*}
}

In particular, $P_{\MP}=P_{\MP}^{T},$ is symmetric, whereas $P_{\widehat D}\ne P_{\widehat D}^{T}.$
\end{theorem}

\begin{proof}
Let $d_{n,k}=[x^n]g f^k$.  From $f=x(1+f+f^2)$ we have $f^k=x f^{k-1}+x f^k+x f^{k+1}.$ For $k\ge1$, multiplication by $g$ and coefficient extraction gives $d_{n+1,k}=d_{n,k-1}+d_{n,k}+d_{n,k+1}.$ For the zeroth column, $g=\frac1{1-2x-xf},$
hence $g=1+2xg+xf g.$
Taking coefficients gives $d_{n+1,0}=2d_{n,0}+d_{n,1}.$
These relations are exactly $\overline D_{\MP}=D_{\MP}P_{\MP}$.
 For $\widehat D$, the identity $A(x)=1+xg(x)$, gives $\widehat d_{0,0}=1, \widehat d_{n,0}=d_{n-1,0}\quad(n\ge1),
\widehat d_{n,k}=d_{n-1,k-1}\quad(n,k\ge1).$
Thus $\widehat d_{n+1,0}=d_{n,0}=d_{n-1,0}+d_{n-1,1}+d_{n-1,2}=\widehat d_{n,0}+\widehat d_{n,1}+\widehat d_{n,2},$
for $n\ge1$, with $\widehat d_{1,0}=1$. Similarly, $\widehat d_{n+1,1}=d_{n,0}=\widehat d_{n,0}+\widehat d_{n,1}+\widehat d_{n,2}$ and, for $k\ge2$,
$\widehat{d}_{n+1,k} = d_{n,k-1} = \widehat{d}_{n,k-1} + \widehat{d}_{n,k} + \widehat{d}_{n,k+1}.$ Consequently, $\overline{\widehat{D}} = \widehat{D} P_{\mathcal{C}},$ as claimed. 
\end{proof}
Using Theorem \ref{thm:production}, 
$\widehat D=\begin{pmatrix}
1&0\\
\mathbf m&D_{\MP}
\end{pmatrix},$
where
$
\mathbf m=\begin{pmatrix}
\MP_0\\
\MP_1\\
\MP_2\\
\vdots
\end{pmatrix}.$ 
Let $e_0=(1,0,0,\ldots)^T.$ If $r_n$ denotes the $n$th row of $D_{\MP}$, then $r_{n+1}=r_nP_{\MP},\quad r_0=e_0^T,$ hence
$r_n=e_0^TP_{\MP}^n.$
\begin{proposition}
\label{prop:MP-moment}
For every $n\geq0$, $\MP_n=e_0^TP_{\MP}^ne_0=(P_{\MP}^n)_{0,0}.$
\end{proposition}
\begin{proof}
The first column of $D_{\MP}$ is $(d_{n,0})_{n\geq0}=(\MP_n)_{n\geq0}.$ Since $r_n=e_0^TP_{\MP}^n,$ its zeroth component is
$d_{n,0}=e_0^TP_{\MP}^ne_0.$
\end{proof}
The symmetry of $P_{\MP}$ gives the following factorization.
\begin{theorem}
\label{thm:MP-Cholesky}
The Hankel matrix of the Motzkin-prefix numbers satisfies
$D_{\MP}D_{\MP}^{T}=\left(\MP_{i+j}\right)_{i,j\geq0}.$  Then $\det\left(\MP_{i+j}\right)_{0\leq i,j<n}=1.$
\end{theorem}
\begin{proof}
The $(i,j)$-entry of $D_{\MP}D_{\MP}^{T}$ is $(D_{\MP}D_{\MP}^{T})_{i,j}=r_ir_j^T.$
Using $r_i=e_0^TP_{\MP}^i$ and $r_j^T=(P_{\MP}^j)^Te_0,$ we obtain
\[(D_{\MP}D_{\MP}^{T})_{i,j}=e_0^TP_{\MP}^i(P_{\MP}^j)^Te_0=e_0^TP_{\MP}^{i+j}e_0,\]
because $P_{\MP}$ is symmetric. By
Proposition~\ref{prop:MP-moment}, $e_0^TP_{\MP}^{i+j}e_0=\MP_{i+j}.$

The proof of second part follows directly from first part; see also Krattenthaler and Yaqubi~\cite{KrattenthalerYaqubi2018}.
\end{proof}
\begin{corollary}
\label{cor:MP-hankel}
For every $n\geq1$, $\det\left(\MP_{i+j}\right)_{0\leq i,j<n}=1.$
\end{corollary}
\begin{proof}
The proof follows directly from Theorem~\ref{thm:MP-Cholesky}; see also Krattenthaler and Yaqubi~\cite{KrattenthalerYaqubi2018}.
\end{proof}
Define the Riordan array $R=\left(\frac{1}{1+x+x^2},x\right).$
Its inverse is $R^{-1}=(1+x+x^2,x).$
{\tiny
\begin{align*}
R &= 
\begin{pmatrix}
 1 &  0 &  0 &  0 &  0 &  0 & \cdots \\
-1 &  1 &  0 &  0 &  0 &  0 & \cdots \\
 0 & -1 &  1 &  0 &  0 &  0 & \cdots \\
 1 &  0 & -1 &  1 &  0 &  0 & \cdots \\
-1 &  1 &  0 & -1 &  1 &  0 & \cdots \\
 0 & -1 &  1 &  0 & -1 &  1 & \cdots \\
\vdots & \vdots & \vdots & \vdots & \vdots & \vdots & \ddots
\end{pmatrix},
&
R^{-1} &= 
\begin{pmatrix}
 1 &  0 &  0 &  0 &  0 &  0 & \cdots \\
 1 &  1 &  0 &  0 &  0 &  0 & \cdots \\
 1 &  1 &  1 &  0 &  0 &  0 & \cdots \\
 0 &  1 &  1 &  1 &  0 &  0 & \cdots \\
 0 &  0 &  1 &  1 &  1 &  0 & \cdots \\
 0 &  0 &  0 &  1 &  1 &  1 & \cdots \\
\vdots & \vdots & \vdots & \vdots & \vdots & \vdots & \ddots
\end{pmatrix}.
\end{align*}}
The matrix $R^{-1}$  is a lower-triangular matrix whose
entries are constant along each diagonal parallel to the main diagonal, with entries
\[
(R^{-1})_{j,k}=
\begin{cases}
1,&j-k\in\{0,1,2\},\\
0,&\text{otherwise}.
\end{cases}
\]
Such matrix is called  lower-triangular Toeplitz.
\begin{theorem}
\label{thm:boundary-factorization}
The two Riordan arrays satisfy $\widehat D=D_{\MP}R.$
Equivalently, $(A,f)=(g,f)\left(\frac{1}{1+x+x^2},x\right).$
\end{theorem}
\begin{proof}
By the multiplication rule for ordinary Riordan arrays, $D_{\MP}R=\left(\frac{g(x)}{1+f(x)+f(x)^2},f(x)\right).$
Since $f=x(1+f+f^2),$
we have $1+f+f^2=\frac{f}{x}.$ Therefore $\frac{g(x)}{1+f(x)+f(x)^2}=\frac{xg(x)}{f(x)}.$ Using $xg(x)=A(x)f(x),$
we obtain $\frac{xg(x)}{f(x)}=A(x).$ Hence $D_{\MP}R=(A,f)=\widehat D.$
\end{proof}
\begin{corollary}
\label{cor:local-inverse}
From $D_{\MP}=\widehat DR^{-1},$
we obtain \[d_{n,k}=\sum_{j=k}^{n}\widehat d_{n,j}(R^{-1})_{j,k}=\widehat d_{n,k}+\widehat d_{n,k+1}+\widehat d_{n,k+2},\]
with the convention that $\widehat d_{n,j}=0\quad\text{for }j>n.$ Conversely, since $R=\left(\frac{1}{1+x+x^2},x\right),$
we have
\[
\frac{1}{1+x+x^2}
=
1-x+x^3-x^4+x^6-x^7+x^9-x^{10}+\cdots.
\]
Hence $\widehat D=D_{\MP}R,$
and therefore $\widehat d_{n,k}=\sum_{j=k}^{n}d_{n,j}\,r_{j-k},$
where
\[
r_m=
\begin{cases}
1,&m\equiv0\pmod3,\\
-1,&m\equiv1\pmod3,\\
0,&m\equiv2\pmod3.
\end{cases}
\]
Thus
\[
\widehat d_{n,k}=d_{n,k}-d_{n,k+1}+d_{n,k+3}-d_{n,k+4}+d_{n,k+6}-d_{n,k+7}+\cdots,
\]
where the sum terminates when the index exceeds $n$. Equivalently,
\[
\widehat d_{n,k}=\sum_{\substack{k\leq j\leq n\\j-k\equiv0\pmod3}}d_{n,j}-\sum_{\substack{k\leq j\leq n\\j-k\equiv1\pmod3}}d_{n,j}.
\]

For the first column, since $\widehat d_{n,0}=a_n,$
this becomes
\[a_n=d_{n,0}-d_{n,1}+d_{n,3}-d_{n,4}+d_{n,6}-d_{n,7}+\cdots.
\]
\end{corollary}
\begin{theorem}
\label{thm:production-conjugation}
The production matrices satisfy $P_{\widehat D}=R^{-1}P_{\MP}R.$
\end{theorem}
\begin{proof}
From $\widehat D=D_{\MP}R$, we obtain $\overline{\widehat D}=\overline{D}_{\MP}R.$ Since $\overline{D}_{\MP}=D_{\MP}P_{\MP},$
we have $\overline{\widehat D}=D_{\MP}P_{\MP}R$ On the other hand,
$\overline{\widehat D}=\widehat DP_{\widehat D}=D_{\MP}RP_{\widehat D}.$
Since $D_{\MP}$ is invertible, $P_{\MP}R=RP_{\widehat D},$
and therefore $P_{\widehat D}=R^{-1}P_{\MP}R.$
\end{proof}
\begin{corollary}
\label{cor:production-powers}
For $n,m\geq0$, $P_{\widehat D}^{\,m}=R^{-1}P_{\MP}^{\,m}R,$ and $a_n=(P_{\widehat D}^{\,n})_{0,0}.$
\end{corollary}
\begin{corollary}
\label{cor:strict-endpoint-inequality}
Put $m_{n,j}=(P_{\MP}^{\,n})_{0,j}=d_{n,j}.$
Thus \(m_{n,j}\) is the number of Motzkin paths of length \(n\)
starting at height \(0\) and ending at height \(j\).For every \(n\geq k\geq0\),
\[\sum_{\substack{k\leq j\leq n\\j-k\equiv0\pmod3}}m_{n,j}>\sum_{\substack{k\leq j\leq n\\j-k\equiv1\pmod3}}m_{n,j}.\]
\end{corollary}


For $r \in \{0,1,2\}$, let \(E_r(n) = \sum_{\substack{0\leq j\leq n\\ j\equiv r\pmod3}} d_{n,j}\)
denote the number of Motzkin paths of length $n$ whose final height is congruent to $r$ modulo $3$.

\begin{theorem}
\label{thm:period3-distribution}
For every $n\geq1$, the endpoint totals modulo $3$ form an arithmetic progression with common difference $a_n$:
\begin{equation*}
E_0(n) = 3^{n-1} + a_n, \qquad E_1(n) = 3^{n-1}, \qquad E_2(n) = 3^{n-1} - a_n.
\end{equation*}
\end{theorem}

\begin{proof}
Taking $k=0$ and $k=1$ in Corollary~\ref{cor:local-inverse} yields
\[
E_0(n) - E_1(n) = \widehat d_{n,0} = a_n,\quad E_1(n) - E_2(n) = \widehat d_{n,1} = d_{n-1,0} = a_n,
\]
where we used Lemma~\ref{prop:two-riordan-arrays} and the fact that $d_{n-1,0}=a_n$ for $n\ge 1$. Furthermore, the row sums of $D_{\MP}$ are given by $\frac{g(x)}{1-f(x)} = \frac{1}{1-3x}$, providing the total sum
\begin{equation*}
E_0(n) + E_1(n) + E_2(n) = 3^n.
\end{equation*}
Substituting $E_0(n) = E_1(n) + a_n$ and $E_2(n) = E_1(n) - a_n$ into this total gives $3E_1(n) = 3^n$, from which the exact formulas immediately follow.
\end{proof}
\begin{corollary}
\label{cor:period3-filter}
Let $u_0,u_1,u_2 \in \mathbb{C}$. For every $n\geq1$,
\begin{equation*}
\sum_{j=0}^{n}u_{j\bmod3}\,d_{n,j} = (u_0+u_1+u_2)3^{n-1} + (u_0-u_2)a_n.
\end{equation*}
\end{corollary}

\begin{corollary}
\label{cor:root-unity}
Let $\omega$ be a primitive third root of unity(\(\omega=e^{2\pi i/3}.\))  For every $n\geq1$,
\begin{equation*}
\sum_{j=0}^{n}d_{n,j}\omega^j = (1-\omega^2)a_n \qquad\text{and}\quad \sum_{j=0}^{n}d_{n,j}\omega^{2j} = (1-\omega)a_n.
\end{equation*}
\end{corollary}
\begin{proof}
Apply Corollary~\ref{cor:period3-filter} with $(u_0,u_1,u_2)=(1,\omega,\omega^2)$ and $(1,\omega^2,\omega)$, respectively, using $1+\omega+\omega^2=0$.
\end{proof}
\begin{corollary}\label{cor:row-sums-hatD}
For $n \geq 1$, the row sums of $\widehat D$ satisfy
$
\sum_{k=0}^{n}\widehat{d}_{n,k} = a_n + 3^{n-1} = E_0(n),
$
with $\sum_{k=0}^{0}\widehat{d}_{0,k} = 1$.
\end{corollary}

\begin{proof}
The row-sum generating function of $\widehat D = (A,f)$ is given by
\[
\frac{A(x)}{1-f(x)} = A(x)^2 = A(x) + \frac{x}{1-3x}.
\]
Comparing the coefficients of $x^n$ for $n \geq 1$ yields $\sum_{k=0}^{n}\widehat{d}_{n,k} = a_n + 3^{n-1}$. The relation $\sum_{k=0}^{n}\widehat{d}_{n,k} = E_0(n)$ then follows directly from Theorem~\ref{thm:period3-distribution}.
\end{proof}

\begin{theorem}
\label{thm:explicit-c}
For $n\geq1$ and $0\leq k\leq n$, $\widehat d_{n,k}$ given by
\begin{equation*}
\frac{1}{n} \sum_{m=k}^{n} m\,[t^{n-m}](1+t+t^2)^n = \frac{1}{n} \sum_{m=k}^{n} m \sum_{j=0}^{\lfloor (n-m)/2\rfloor} \binom{n}{j,\ n-m-2j,\ m+j}.
\end{equation*}
\end{theorem}

\begin{proof}
Because $f(x)=x+O(x^2)$, we have $\widehat d_{n,k} = \sum_{m=k}^{n}[x^n]f(x)^m$. Applying Lagrange inversion to $f(x)=x(1+f(x)+f(x)^2)$ gives
\begin{equation*}
[x^n]f(x)^m = \frac{m}{n}[t^{n-m}](1+t+t^2)^n \qquad (1\le m\le n).
\end{equation*}
Expanding $(1+t+t^2)^n$ yields $[t^{n-m}](1+t+t^2)^n = \sum_{j=0}^{\lfloor (n-m)/2\rfloor} \binom{n}{j,\ n-m-2j,\ m+j}$. Summing over $m$ proves both identities.
\end{proof}
For a proper Riordan array $D = (g(x), f(x))$, its inverse is $D^{-1} = \left(1/g(\bar f(x)), \bar f(x)\right)$, where $\bar f$ is the compositional inverse of $f$.
\begin{proposition}
\label{prop:inverse-riordan}
The inverse Riordan arrays are
\begin{equation*}
D_{\MP}^{-1} = \left( \frac{1-x}{1+x+x^2}, \frac{x}{1+x+x^2} \right)
\quad\text{and}\quad
\widehat D^{-1} = \left( 1-x, \frac{x}{1+x+x^2} \right).
\end{equation*}

\end{proposition}

\begin{proof}
From $f = x(1+f+f^2)$, its compositional inverse is $\bar f(x) = \frac{x}{1+x+x^2}$. For $D_{\MP}=(g,f)$ with $g(x)=\frac{1}{1-2x-xf(x)}$, using $f(\bar f(x))=x$ yields
\begin{equation*}
\frac{1}{g(\bar f(x))} = 1 - 2\bar f(x) - x\bar f(x) = \frac{1-x}{1+x+x^2},
\end{equation*}
establishing $D_{\MP}^{-1}$. Next, using the boundary factorization $\widehat D = D_{\MP}R$ with $R = \left(\frac{1}{1+x+x^2}, x\right)$ and $R^{-1} = (1+x+x^2, x)$, Riordan multiplication gives
\begin{equation*}
\widehat D^{-1} = R^{-1} D_{\MP}^{-1} = (1+x+x^2, x) \left( \frac{1-x}{1+x+x^2}, \frac{x}{1+x+x^2} \right) = \left( 1-x, \frac{x}{1+x+x^2} \right). \qedhere
\end{equation*}
\end{proof}
Using Proposition~\ref{prop:inverse-riordan}, we have
\begin{align*}
D_{\MP}^{-1} = 
\begin{pmatrix}
1 & 0 & 0 & 0 & 0 & \cdots \\
-2 & 1 & 0 & 0 & 0 & \cdots \\
1 & -3 & 1 & 0 & 0 & \cdots \\
1 & 3 & -4 & 1 & 0 & \cdots \\
-2 & 1 & 6 & -5 & 1 & \cdots \\
\vdots & \vdots & \vdots & \vdots & \ddots & \ddots
\end{pmatrix}, \quad
\widehat D^{-1} = 
\begin{pmatrix}
1 & 0 & 0 & 0 & 0 & \cdots \\
-1 & 1 & 0 & 0 & 0 & \cdots \\
0 & -2 & 1 & 0 & 0 & \cdots \\
0 & 1 & -3 & 1 & 0 & \cdots \\
0 & 1 & 3 & -4 & 1 & \cdots \\
\vdots & \vdots & \vdots & \vdots & \ddots & \ddots
\end{pmatrix}.
\end{align*}
\begin{corollary}
\label{cor:first-column-inverse-MP}
The first column of $D_{\MP}^{-1}$ has generating function $\frac{1-x}{1+x+x^2}$, giving the period-3 sequence
\begin{equation*}
(D_{\MP}^{-1})_{n,0} =
\begin{cases}
1, & n\equiv 0 \pmod 3, \\
-2, & n\equiv 1 \pmod 3, \\
1, & n\equiv 2 \pmod 3.
\end{cases}
\end{equation*}
\end{corollary}

\begin{proof}
The first column of $D_{\MP}^{-1} = \left( \frac{1-x}{1+x+x^2}, \frac{x}{1+x+x^2} \right)$ has generating function
\begin{equation*}
\frac{1-x}{1+x+x^2} = \frac{(1-x)^2}{1-x^3} = (1-2x+x^2)\sum_{m\ge 0} x^{3m},
\end{equation*}
from which extracting the coefficient of $x^n$ completes the proof. \qedhere
\end{proof}
Since $P_{\MP}$ is symmetric tridiagonal with diagonal parameters $\alpha_0=2$, $\alpha_n=1$ ($n\ge 1$) and subdiagonal parameters $\beta_n=1$ ($n\ge 1$), the row polynomials $Q_n(t) = \sum_{k=0}^{n} (D_{\MP}^{-1})_{n,k} t^k$ form a monic orthogonal polynomial sequence defined by $Q_0(t)=1$, $Q_1(t)=t-2$, and
\begin{equation*}
Q_{n+1}(t) = (t-1)Q_n(t) - Q_{n-1}(t) \qquad (n\ge 1).
\end{equation*}

\begin{remark}
\label{rem:MP-jfraction}
These parameters yield the Jacobi continued fraction expansion for $g(x) = \sum_{n\ge 0} \MP_n x^n$:
\begin{equation*}
g(x) = \cfrac{1}{1-2x-\cfrac{x^2}{1-x-\cfrac{x^2}{1-x-\ddots}}}.
\end{equation*}
Consequently, the standard J-fraction determinant formula gives
\begin{equation*}
\det_{0\le i,j < n} (\MP_{i+j}) = \MP_0^n \prod_{j=1}^{n-1} \beta_j^{n-j} = 1 \qquad (n\ge 1),
\end{equation*}
which is another proof of results in Corollary ~\ref{cor:MP-hankel}.
\end{remark}

\begin{corollary}
\label{cor:cheb}
For $n\ge 1$, $Q_n(t) = U_n\left(\frac{t-1}{2}\right) - U_{n-1}\left(\frac{t-1}{2}\right)$, where $U_n$ denotes the Chebyshev polynomial of the second kind.
\end{corollary}

\begin{remark}
\label{rem:C-not-orthogonal}
Because $P_{\widehat D}$ is not tridiagonal, $\widehat D^{-1}$ does not form an orthogonal polynomial sequence.
\end{remark}
\subsection{A unified factorization of shifted Hankel matrices}
\label{sec:unified-hankel-factorization}
Let $\widehat D = \{a_n\}_{n\geq0}$ be sequence \href{https://oeis.org/A005773}{OEIS A005773}, and let $H^{(r)} := (a_{i+j+r})_{i,j\geq0}$ denote its $r$-th shifted Hankel matrix. Put $T := P - e_0 e_0^T$, where $e_0 = (1,0,0,\ldots)^T$ and $P = P_{\MP} = P^T$ be the symmetric production matrix of $D_{\MP}$, means
\begin{equation*}
P = \begin{pmatrix} 2 & 1 & 0 & \cdots \\ 1 & 1 & 1 & \cdots \\ 0 & 1 & 1 & \ddots \\ \vdots & \vdots & \ddots & \ddots \end{pmatrix}, 
\qquad
T = \begin{pmatrix} 1 & 1 & 0 & \cdots \\ 1 & 1 & 1 & \cdots \\ 0 & 1 & 1 & \ddots \\ \vdots & \vdots & \ddots & \ddots \end{pmatrix}.
\end{equation*}
Define the \textit{the block-diagonal matrix} $\oplus$ as  $1\oplus T=\begin{pmatrix}1&0\\0&T\end{pmatrix},$
that is, the scalar $1$ occupies the $(0,0)$ position and $T$ occupies
the remaining rows and columns.  Accordingly, $(1\oplus T)_n=1\oplus T_{n-1}.$
\begin{theorem}
\label{thm:unified-factorization}
Let $H^{(r)} = (a_{i+j+r})_{i,j\geq0}$ ($r\geq0$) be the shifted Hankel matrices of sequence $\widehat D=\{a_n\}_{n\geq0}$. Let $P=P_{\MP}$ denote the symmetric production matrix of $D_{\MP}$, $e_0=(1,0,0,\ldots)^T$, and $T=P-e_0e_0^T$. Then
\begin{equation*}
H^{(0)} = \widehat D(1\oplus T)\widehat D^{\,T}
\qquad\text{and}\qquad
H^{(r)} = D_{\MP}P^{\,r-1}D_{\MP}^{\,T} \quad (r\geq1).
\end{equation*}
\end{theorem}

\begin{proof}
Let $r_n = e_0^T P^n$ denote the $n$th row of $D_{\MP}$. Since $P=P^T$, Proposition~\ref{prop:MP-moment} gives $e_0^T P^m e_0 = \MP_m$. 
For $r\geq1$, the $(i,j)$ entry of $D_{\MP}P^{r-1}D_{\MP}^T$ is
\begin{equation*}
r_i P^{r-1} r_j^T = e_0^T P^i P^{r-1} (P^j)^T e_0 = e_0^T P^{i+j+r-1} e_0 = \MP_{i+j+r-1} = a_{i+j+r},
\end{equation*}
where $a_n = \MP_{n-1}$ for $n\geq1$. Thus, $H^{(r)} = D_{\MP}P^{r-1}D_{\MP}^T$.
For $r=0$, let $\widehat r_n$ denote the $n$th row of $\widehat D$. The identity $xg(x) = A(x)f(x)$ yields $\widehat d_{n,0} = a_n = d_{n-1,0}$ and $\widehat d_{n,k} = d_{n-1,k-1}$ for $n,k\geq1$, so $\widehat r_n = \bigl( (r_{n-1})_0, r_{n-1} \bigr)$ for $n\geq1$. Since $1\oplus T = \begin{pmatrix} 1 & 0 \\ 0 & T \end{pmatrix}$ and $T = P - e_0e_0^T$, for $i,j\geq1$ we obtain
\begin{align*}
\widehat r_i(1\oplus T)\widehat r_j^T &= (r_{i-1})_0(r_{j-1})_0 + r_{i-1}Tr_{j-1}^T\\
& = r_{i-1}Pr_{j-1}^T \\
&= e_0^T P^{i+j-1} e_0 \\
&= \MP_{i+j-1} = a_{i+j}.
\end{align*}
When $i=0$ or $j=0$, as $\widehat r_0 = e_0^T$, we directly get $\widehat r_0(1\oplus T)\widehat r_j^T = \widehat d_{j,0} = a_j = a_{0+j}$ and $\widehat r_0(1\oplus T)\widehat r_0^T = 1 = a_0$. Therefore, $(\widehat D(1\oplus T)\widehat D^T)_{i,j} = a_{i+j}$ for all $i,j\geq0$, establishing $H^{(0)} = \widehat D(1\oplus T)\widehat D^T$. \qedhere
\end{proof}
\begin{remark}
For $r\geq1$, $H^{(r)}$ is generated by powers of $P$. The case $r=0$ is exceptional, inserting the boundary matrix $1\oplus T$ to record the correction distinguishing $\widehat D$ from the shifted Motzkin-prefix sequence.
\end{remark}
Let $\Delta_n^{(r)} := \det(a_{i+j+r})_{0\leq i,j<n}$ with $\Delta_0^{(r)}=1$. Let $P_n$ denote the leading $n\times n$ principal submatrix of
the infinite matrix $P$.
\begin{corollary}
\label{cor:unified-hankel-determinants}
For every $n\geq1$,
\begin{equation*}
\Delta_n^{(0)} = \det T_{n-1}
\quad\text{and}\quad
\Delta_n^{(r)} = \det\left((P^{r-1})_n\right) \quad (r\geq1).
\end{equation*}
Here $(P^s)_n$ denotes the leading $n\times n$ principal
submatrix of the infinite matrix $P^s$.
\end{corollary}

\begin{proof}
Taking $n\times n$ principal submatrices in Theorem~\ref{thm:unified-factorization} gives $H^{(0)}_n = \widehat D_n (1\oplus T_n) \widehat D_n^T$ and $H^{(r)}_n = D_{\MP,n} (P^{r-1})_n D_{\MP,n}^T$. Since $\det \widehat D_n = \det D_{\MP,n} = 1$, taking determinants yields $\Delta_n^{(0)} = \det(1\oplus T_{n-1}) = \det T_{n-1}$ and $\Delta_n^{(r)} = \det((P^{r-1})_n)$.
\end{proof}
If $\delta_n=\det P_n$, then $\delta_n=\delta_{n-1}-\delta_{n-2}\quad(n\geq2),$ with $\delta_0=1,\qquad \delta_1=2.$
Consequently,
\[\Delta_n^{(2)}=\delta_n=2\cos\left(\frac{(n-1)\pi}{3}\right).
\]
More generally,  $\Delta_n^{(r)}=\det\left((P^{r-1})_n\right),\quad r\geq1,$
while
\[
\Delta_n^{(0)}
=
\det\left((1\oplus T)_n\right),
\qquad
T=P-e_0e_0^T.
\]
Thus the complete family of shifted Hankel determinants of
$\widehat D$ is encoded by the single symmetric production matrix
$P$ together with the boundary correction $1\oplus T$.

The classical \emph{Desnanot--Jacobi identity} (also known as \emph{Lewis Carroll's condensation formula}) \cite{Desnanot1819, Krattenthaler2005} states that for any $n \times n$ matrix $M = (m_{i,j})_{1 \le i,j \le n}$ with $n \ge 3$,
\begin{equation}\label{eq:desnanot-jacobi}
\det(M) \det(M_{1,n; 1,n}) = \det(M_{1;1}) \det(M_{n;n}) - \det(M_{1;n}) \det(M_{n;1}),
\end{equation}
where $M_{i;j}$ denotes the $(n-1) \times (n-1)$ submatrix of $M$ obtained by removing row $i$ and column $j$, and $M_{i,j; k,l}$ denotes the $(n-2) \times (n-2)$ submatrix obtained by removing rows $i,j$ and columns $k,l$.
\begin{theorem}
\label{thm:four-shifts}
For $
\Delta_n^{(r)}=\det(a_{i+j+r})_{0\leq i,j<n},$ we have
\[
\Delta_n^{(3)}
=
\begin{cases}
2n+1,&n\equiv0\pmod3,\\[2mm]
2n+3,&n\equiv1\pmod3,\\[2mm]
2n+2,&n\equiv2\pmod3.
\end{cases}
\]
\end{theorem}
\begin{proof}
The base cases are $\Delta_0^{(3)} = 1$ and $\Delta_1^{(3)} = 5$. For $n \geq 2$, applying the Desnanot--Jacobi identity to the Hankel matrix $(a_{i+j+1})_{0 \leq i,j \leq n}$ yields
\[
\Delta_{n+1}^{(1)}\Delta_{n-1}^{(3)} = \Delta_n^{(1)}\Delta_n^{(3)} - \left(\Delta_n^{(2)}\right)^2.
\]
Since $\Delta_m^{(1)} = 1$ for all $m \geq 0$, this simplifies to $\Delta_n^{(3)} = \Delta_{n-1}^{(3)} + \left(\Delta_n^{(2)}\right)^2$. Unfolding the recurrence with $\Delta_1^{(3)} = 1 + (\Delta_1^{(2)})^2$ gives
\[
\Delta_n^{(3)} = 1 + \sum_{m=1}^{n} \left(\Delta_m^{(2)}\right)^2, \qquad n \geq 1.
\]
Using the explicit evaluation $\Delta_m^{(2)} = 2\cos\left(\frac{(m-1)\pi}{3}\right)$, its square is periodic modulo $3$:
\[
\left(\Delta_m^{(2)}\right)^2 = \begin{cases}
4, & m \equiv 1 \pmod 3, \\
1, & m \equiv 0, 2 \pmod 3.
\end{cases}
\]
Summing these values explicitly across residue classes yields
\[
\Delta_n^{(3)} = 1 + n + 3\,\#\{1 \leq m \leq n : m \equiv 1 \pmod 3\},
\]
which evaluates directly to $2n+1$ if $n \equiv 0 \pmod 3$, $2n+3$ if $n \equiv 1 \pmod 3$, and $2n+2$ if $n \equiv 2 \pmod 3$.
\end{proof}
\begin{lemma}
\label{lem:higher-shift-DJ}
For every $r\geq0$ and every $n\geq2$,
\begin{equation}
\Delta_n^{(r)}\Delta_{n-2}^{(r+2)}=\Delta_{n-1}^{(r)}\Delta_{n-1}^{(r+2)}-\left(\Delta_{n-1}^{(r+1)}\right)^2,
\label{eq:higher-shift-DJ}
\end{equation}
\end{lemma}
\begin{proof}
The result follows from a straightforward application of the Desnanot--Jacobi identity.
\end{proof}
\begin{theorem}
\label{thm:D4-explicit}
For every $n\geq0$, the fourth shifted Hankel determinant $\Delta_n^{(4)} = \det(a_{i+j+4})_{0\leq i,j<n}$ is given by
\begin{equation*}
\Delta_n^{(4)} = \begin{cases}
\frac{2n}{3}+1, & n\equiv0\pmod6, \\[2mm]
2n^2+\frac{19n}{3}+\frac{14}{3}, & n\equiv1\pmod6, \\[2mm]
2n^2+\frac{17n}{3}+\frac{11}{3}, & n\equiv2\pmod6, \\[2mm]
-\frac{2n}{3}-1, & n\equiv3\pmod6, \\[2mm]
-2n^2-\frac{19n}{3}-\frac{14}{3}, & n\equiv4\pmod6, \\[2mm]
-2n^2-\frac{17n}{3}-\frac{11}{3}, & n\equiv5\pmod6.
\end{cases}
\end{equation*}
\end{theorem}
\begin{proof}
For $n \geq 2$, applying the Desnanot--Jacobi identity (with $r=2$) yields the recurrence
\begin{equation}\label{eq:DJ_recurrence}
\Delta_n^{(2)}\Delta_{n-2}^{(4)} = \Delta_{n-1}^{(2)}\Delta_{n-1}^{(4)} - \left(\Delta_{n-1}^{(3)}\right)^2.
\end{equation}
Since $\Delta_n^{(2)} = 2\cos\left(\frac{(n-1)\pi}{3}\right) \in \{\pm 1, \pm 2\}$, it is non-zero for all integers $n$. Thus, recurrence~\eqref{eq:DJ_recurrence} uniquely determines the sequence $\left(\Delta_n^{(4)}\right)_{n \ge 0}$ from the initial values $\Delta_0^{(4)} = 1$ and $\Delta_1^{(4)} = a_4 = 13$.

To establish the explicit formulas, we analyze the sequence across its six residue classes modulo $6$. Let $F_s(k) := \Delta_{6k+s}^{(4)}$ for $0 \leq s \leq 5$. Evaluating recurrence~\eqref{eq:DJ_recurrence} according to the residue class of $n$ generates a system of six interlaced identities:
\begin{align*}
F_0(k)  &= 2F_1(k) - (12k+5)^2, \\
F_1(k)  &= -F_2(k) + (12k+6)^2, \\
2F_2(k) &= F_3(k) + (12k+7)^2, \\
F_3(k)  &= 2F_4(k) + (12k+11)^2, \\
F_4(k)  &= -F_5(k) - (12k+12)^2, \\
2F_5(k) &= F_0(k+1) - (12k+13)^2.
\end{align*}
Direct substitution confirms that the proposed polynomial expressions for $F_s(k)$ identically satisfy this system. Furthermore, evaluating these polynomials at $k=0$ gives the initial tuple
\[
\big(F_0(0), F_1(0), F_2(0), F_3(0), F_4(0), F_5(0)\big) = (1, 13, 23, -3, -62, -82),
\]
which precisely matches the direct determinant computations of $\Delta_n^{(4)} = \det(a_{i+j+4})_{0\leq i,j<n}$ for $0 \leq n \leq 5$. Because the proposed formulas satisfy both the uniquely defining recurrence and the correct initial conditions, they hold for all $n \geq 0$.
\end{proof}

\section{Known identities and combinatorial interpretations}
\label{sec:further-connections}

We now collect several identities and combinatorial interpretations of the sequence $(a_n)_{n \geq 0}$. These formulations are compatible with our Riordan-array notation and often correspond to properties recorded for the OEIS sequence \href{https://oeis.org/A005773}{A005773}.


\begin{proposition}\label{prop:A-quadratic}
The generating function $A(x)$ satisfies the quadratic equation $A(x)^2 - A(x) = \frac{x}{1-3x}$. Extracting coefficients yields the convolution identity $\sum_{k=0}^{n} a_k a_{n-k} = a_n + 3^{n-1}$ for $n \geq 1$. Separating the boundary terms gives the recurrence $a_n = 3^{n-1} - \sum_{k=1}^{n-1} a_k a_{n-k}$.
\end{proposition}


\begin{proposition}\label{prop:motzkin-composition}
Let $M(x) = \sum_{n \geq 0} M_n x^n$ be the generating function of the Motzkin numbers. Then $A(x) = \frac{1}{1-xM(x)}$, or equivalently, $A(x) = 1 + x M(x) A(x)$. This implies $a_n = \sum_{k=0}^{n-1} M_k \, a_{n-1-k}$ for $n \geq 1$. Moreover, $g(x) = M(x)A(x)$.
\end{proposition}


\begin{proposition}\label{prop:weighted-compositions}
For $n \geq 1$, we have $a_n = \sum_{\lambda \models n} \prod_{j=1}^{\ell(\lambda)} M_{\lambda_j-1}$, where the sum runs over all ordered compositions $\lambda = (\lambda_1, \ldots, \lambda_{\ell(\lambda)})$ of $n$. The empty composition yields $a_0 = 1$.
\end{proposition}

\begin{proof}
By Proposition~\ref{prop:motzkin-composition}, $A(x) = \big(1-\sum_{k \geq 1} M_{k-1} x^k\big)^{-1}$. Expanding this geometrically and extracting the coefficient of $x^n$ directly yields the sum over compositions, where each part of size $k$ contributes a weight of $M_{k-1}$.
\end{proof}


\begin{proposition}\label{prop:central-binomial}
For every $n \geq 0$, $\MP_n = a_{n+1} = \sum_{k=0}^{n} \binom{n}{k} \binom{k}{\lfloor k/2 \rfloor}$. Thus, the Motzkin-prefix sequence is the binomial transform of \href{https://oeis.org/A001405}{OEIS A001405}.
\end{proposition}

\begin{proof}
Let $c_k = \binom{k}{\lfloor k/2 \rfloor}$ and $C(x)$ be its generating function. The ordinary binomial transform of $(c_k)_{k \geq 0}$ is $b_n = \sum_{k=0}^{n} \binom{n}{k} c_k$, with generating function $\frac{1}{1-x} C\big( \frac{x}{1-x} \big)$. For \href{https://oeis.org/A005773}{OEIS A005773}, this explicitly evaluates to $b_n = a_{n+1} = \MP_n$.
\end{proof}


\begin{proposition}\label{prop:inverse-binomial-A001700}
Let $b_k = \binom{2k+1}{k+1}$ (\href{https://oeis.org/A001700}{OEIS A001700}). Then $a_{n+1} = \sum_{k=0}^{n} (-1)^{n-k} \binom{n}{k} \binom{2k+1}{k+1}$ for $n \geq 0$.
\end{proposition}

\begin{proof}
This follows directly from applying the binomial inverse transform, $c_n = \sum_{k=0}^{n} (-1)^{n-k} \binom{n}{k} b_k$, to the sequence $b_k$. For \href{https://oeis.org/A005773}{OEIS A005773}, this evaluates to $c_n = a_{n+1}$.
\end{proof}


\begin{proposition}\label{prop:catalan}
Let $C_k = \frac{1}{k+1} \binom{2k}{k}$ be the Catalan numbers, with generating function $C(z) = \frac{1-\sqrt{1-4z}}{2z}$. Then $A(x) = 1 + \frac{x}{1-3x} C\big( -\frac{x}{1-3x} \big)$. Consequently, for $n \geq 1$, $a_n = \sum_{k=0}^{n-1} (-1)^k 3^{n-1-k} \binom{n-1}{k} C_k$, and equivalently $a_{n+1} = \sum_{k=0}^{n} (-1)^k 3^{n-k} \binom{n}{k} C_k$.
\end{proposition}

\begin{proof}
Substituting $y = \frac{x}{1-3x}$ into the identity $1 + y C(-y) = \frac{1}{2} \big( 1 + \sqrt{1+4y} \big)$ simplifies to $\frac{1}{2} \big( 1 + \sqrt{\frac{1+x}{1-3x}} \big) = A(x)$. Rearranging yields $\frac{A(x)-1}{x} = \sum_{k \geq 0} (-1)^k C_k x^k (1-3x)^{-k-1}$. Applying the negative binomial expansion $(1-3x)^{-k-1} = \sum_{m \geq 0} \binom{m+k}{k} 3^m x^m$ and extracting the coefficient of $x^n$ completes the proof.
\end{proof}


\begin{proposition}\label{prop:base3-trinomial}
For $n \geq 1$, $a_n$ counts the $n$-digit integers in base $3$ with no leading zero whose digit sum equals $n$. Algebraically, $a_n = [t^n] (t+t^2)(1+t+t^2)^{n-1}$. This simplifies to the binomial sum $a_n = \sum_{k=0}^{\lfloor n/2 \rfloor} \binom{n-1}{k} \binom{n-k}{k}$.
\end{proposition}

\begin{proof}
The digit sum condition $\sum d_i = n$ for $d_1 \in \{1,2\}$ and $d_i \in \{0,1,2\}$ ($i \geq 2$) is encoded by the polynomial $(t+t^2)(1+t+t^2)^{n-1}$. Splitting the leading factor gives $a_n = [t^{n-1}](1+t+t^2)^{n-1} + [t^{n-2}](1+t+t^2)^{n-1}$. Combining these adjacent coefficients yields the stated binomial sum.
\end{proof}


\begin{proposition}\label{prop:ordered-trees}
For $n \geq 0$, $a_n$ counts plane rooted trees with $n$ edges in which every non-root vertex has outdegree at most $2$.
\end{proposition}

\begin{proposition}\label{prop:symmetric-dyck}
For $n \geq 1$, $a_n$ counts symmetric Dyck paths of semilength $2n-1$ (or $2n$) having no peaks at an even level.
\end{proposition}

\begin{proposition}\label{prop:dyck-avoiding}
For $n \geq 1$, $a_n$ is the number of Dyck words of semilength $n$ avoiding the factor $DUUU$.
\end{proposition}

\begin{proposition}\label{prop:avoiding-permutations}
For $n \geq 0$, $a_n$ is the number of permutations of $[n]$ avoiding both vincular patterns $1\!-\!23\!-\!4$ and $1\!-\!3\!-\!2$.
\end{proposition}

\begin{proposition}\label{prop:directed-animals}
The sequence $(a_n)_{n \geq 0}$ enumerates directed animals of size $n$ in standard position. Established bijections link these structures to grand Dyck paths and specific multiset configurations \cite{BarilBevanKirgizov2020}.
\end{proposition}

\begin{proposition}\label{prop:multiset-models}
For $n \geq 1$, $a_n$ is the number of $n$-multisets of $[n]$ containing no pair of consecutive integers. Equivalently, $a_n$ is the number of $n$-multisets of $[n]$ in which no integer other than $n$ occurs exactly once. This is closely related to the structures examined in \cite{BarilBevanKirgizov2020}.
\end{proposition}


\begin{proposition}\label{prop:A-asymptotic}
As $n \to \infty$, we have the asymptotic approximation $a_n \sim \frac{3^{n-1/2}}{\sqrt{\pi n}}$. In particular, $\lim_{n \to \infty} \frac{a_{n+1}}{a_n} = 3$.
\end{proposition}

\end{document}